\documentclass{article}
\usepackage{amsmath,amsthm,amssymb, bbm, array, graphicx, tikz}
\usepackage{mathptmx,graphicx}
\usepackage{bm}

\usepackage[margin=1in]{geometry}

\newenvironment{tpic}{\begin{tikzpicture}}{\end{tikzpicture}}

\newtheorem*{theorem*}{Theorem}
\newtheorem{theorem}{Theorem}[section]
\newtheorem{proposition}[theorem]{Proposition}
\newtheorem{corollary}[theorem]{Corollary}
\newtheorem{lemma}[theorem]{Lemma}
\newtheorem{remark}[theorem]{Remark}

\newcommand{\cref}[1]{Corollary~\textup{\ref{#1}}}

\newcommand{\M}{\textmd{MCG}}

\newcommand{\C}{\mathbb{C}}

\newcommand{\gl}{\mathfrak{gl}}
\newcommand{\Sl}{\mathfrak{sl}}
\newcommand{\g}{\mathfrak{g}}
\newcommand{\h}{\mathfrak{h}}
\newcommand{\Hom}{\textmd{Hom}}

\title{On TQFT representations of mapping class groups with boundary}
\author{Greg Kuperberg; Shuang Ming}
\date{}

\begin{document}
 \maketitle

\begin{abstract}
We study the TQFT mapping class group representations for surfaces with boundary associated with the $SU(2)$ gauge group, or equivalently the quantum group $U_q(\Sl(2))$.  We show that at a prime root of unity,
these representations are all irreducible.   We also examine braid
group representations for transcendental values of the quantum
parameter, where we show that the image of every mapping class group
is Zariski dense.
\end{abstract}
\section{Introduction}

Given a integer $p\ge 3$, $A$ a primitive $4p$-th root of unity, we can associate a vector space to a surface with its boundary colored by nonnegative integers less than $p-1$. The mapping class group of the surface act on the vector space projectively. These representations are called $SU(2)$-TQFT representations and have many good properties \cite{FLW1} and it is used to study mapping class groups and provide interesting examples in both geometric topology \cite{KR16} and quantum algebra \cite{GM17}.

In this paper, we are studying the image of mapping class groups in their $SU(2)$-TQFT representations. Usually we understand this question in the following steps: We start with irreduciblity, so we can restrict the image into blocks. Then we check the (in)finiteness of the image. If the image is finite then we have a finite quotient(e.g. Weil representation). Otherwise, we study its closure under Zariski topology or Analytic topology.

The irreducibility problem is first studied by Roberts \cite{R}. Roberts proved when $p$ is an odd prime, the representations are irreducible for closed surfaces. After that, Korinman \cite{Kor2} studied non-prime case: He showed that the representations are irreducible when $p$ is the product of two distinct odd primes and when $p$ is square of an odd prime, and for some other $p$, decomposable examples have been given. In \cite{GG}, they first studied the cases that are with boundary. They proved when $p$ is an odd prime, the representations of any surfaces with one boundary component is irreducible. Recently, Koberda and Santharoubane  \cite{KR17} considered the cases that are with boundaries, and they proved that for any $p$, if one of the boundaries is colored by $1$, the representation can be shown irreducible.

The (in)finiteness problem is first studied by Funar\cite{Funar:mapping}. Funar showed that when $p\ne 2, 3, 4, 6, 10$, the image of $B_3$ in $PSL(V_{0, 3; 1, 1, 1, 3})$ is infinite. It implies for all representations containing this vector space as a tensor-factor of a subspace, the images of the mapping class groups will be infinite. After that, in \cite{Kor}, Korinman gives a criterion for (in)finiteness, which can be checked by hand for surfaces with low genus and small number of boundary components. Especially, Korinman gave the computation for the case $p$ is odd prime, $g=1$ and $b=1$.

By the results of Funar and Korinman, in most cases, the image of the representation is infinites. The closure problem is studied by Freedman, Larsen, Wang \cite{FLW}\cite{LW:so3} and the first author \cite{Ku}. In \cite{FLW} and \cite{Ku}, The Jones representation of braided group is proved dense in $PSU(N)$ for $p$ not equal to $2, 3, 4, 6, 10$. In \cite{Ku}, the first author also proved when $A$ is generic, the image of the braid group is Zariski dense in $PSL(N)$. After that, In \cite{LW:so3}, Larsen and Wang showed when $p>3$ is an odd prime. The $SO(3)$-TQFT representations have dense image when genus $g\ge 2$.

It worth noting that in our setting, Jones representations of $B_n$ can be seen as a special case of $SU(2)$-TQFT representations of $\M(\Sigma_{0, n+1})$. By taking the surface to be a sphere with $(n+1)$ circle boundaries, and $n$ of its boundary components are colored by $1$.

In this paper, we study the case that surfaces are with nonempty boundary. Let $V_{g, b; \vec{c}}$ be the representation associated to the surface of genus $g$ and $b$ boundary components colored by $\vec{c}$(the precise definition will be given in the next section). We proved the following irreducibility theorem.

\begin{theorem}
$V_{g, b; \vec{c}}$ is irreducible at odd prime level under the action of $\M(\Sigma_{g, b})$ for all admissible color $\vec{c}$.
\end{theorem}

When $A$ is not a root of unity, the TQFT is not well defined because the dimension of $V_{g, b; \vec{c}}$ will be infinite when $g\ge 1$. However, when $g=0$, this representation is well-defined because the color set is bounded by the half of the total colors on the boundary. We have the following denseness theorem which generalize an earlier result \cite{Ku} of the first author.

\begin{theorem}
\label{dense}Let $A$ be a transcendental complex number. The image of $\widetilde{PB_n}$ is Zariski dense in the algebraic group $SL(V_{0, n; \vec{c}})$.
\end{theorem}

The structure of this paper is the following. In section 2, we discuss the background for $SU(2)$-TQFT representations. In section 3 to 6, we prove the irreducibility theorem. In section 7, we devolope a set of tools for studying group homomorphisms. In section 8, we prove the Zariski denseness theorem.

\section{Background}
In this section, we give a short introduction to the background of $SU(2)$-TQFT representations. One can find full explanation in \cite{BHMV}.

\subsection{The space $V_{p, g, b; \vec{c}}$}
Given an integer $p\ge 3$, let $A\in \C$ be a primitive 4p-th root of unity. Using the Kauffman skein relation(see picture below), we can associate a link $L\in S^3$ an complex number $<L>_p$.

$$\begin{tpic}[baseline=-.5ex, scale=0.5]
\draw[line width=1pt] (-1,1) -- (1,-1); \draw[line width=1pt] (-1,-1) -- (-0.1,-0.1); \draw[line width=1pt] (0.1,0.1) -- (1,1);
\end{tpic}
=
-A
\begin{tpic}[baseline=-.5ex, scale=0.5]
\draw[line width=1pt] (-1,-1) arc (-45:45:1.414); \draw[line width=1pt] (1,1) arc (135:225:1.414);
\end{tpic}
-A^{-1}
\begin{tpic}[baseline=-.5ex, scale=0.5]
\draw[line width=1pt] (1,-1) arc (45:135:1.414); \draw[line width=1pt] (-1,1) arc (225:315:1.414);
\end{tpic}; \hspace{10mm}
\begin{tpic}[baseline=-.5ex, scale=0.5]
\draw (0,0)[line width=1pt] circle (1);
\end{tpic}
= -A^{2}-A^{-2}.$$

Choose a handlebody of genus $g$ and let $C_g$ be the vector space spanned by the isotopy classes of framed links together with empty links. We fix two of such genus $g$ handlebody $H_g^1$, $H_g^2$ and a glueing:
$$H_g^1\bigcup_{\partial H_g^1 \rightarrow\partial H_g^2} H_g^2 \cong S^3.$$
This gives us a bilinear form
$$(.,.)_{g, p}^H: C_g \times C_g \rightarrow \C$$
by composing with the evaluation $<>_p$ in $S^3$. We define the space $V_{p, g}$ as $C_g$ mod out by the kernel of $<.,.>$.
It was proved in \cite{BHMV} that the spaces $V_{p, g}$ are finite-dimensional, and a set of basis can be given as following.

We choose a trivalent graph $\Gamma$ such that the handlebody $H_g$ retracts to $\Gamma$, and color the edges from the set $\{0, 1,..,p-2\}$. We say a coloring is \textit{admissible} if at each trivalent vertex, the color of the adjacent edges $(i, j, k)$ satisfies:
\begin{enumerate}
\item[i)] $|i-j|\le k\le i+j$,
\item[ii)] $i+j+k\le 2p-4$,
\item[iii)] $i+j+k=0$ mod $2$
\end{enumerate}
A representative of such a colored graph in $C_g$ is given by Jone-Wenzl's idempotent. At each edges colored by $c$, we put $c$ parallel string with an idempotent $f_c$. At vertices colored by $(i, j, k)$, we connect the projectors with the link below.
$$
\begin{tpic}[line width=1pt, baseline=-.5ex]
\draw(0,0)--node[right]{$i$} (0,1);
\draw(0,0)--node[above]{$j$} (-0.87, -0.5);
\draw(0,0)--node[above]{$k$} (0.87, -0.5);
\end{tpic}
=
\begin{tpic}[line width=1pt, baseline=-.5ex]
\draw(0,1) node[rectangle, draw] (x) {$f_i$};
\draw(-0.87, -0.5) node (y)[rectangle, draw]{$f_j$};
\draw(0.87, -0.5) node (z)[rectangle, draw]{$f_k$};
\draw(-1.4, -1)--node[above]{$j$}(y);
\draw(1.74, -1) --node[above]{$k$}(z);
\draw(0,2)--node[right]{$i$}(x);
\draw(x) --node[right]{$\frac{i-j+k}{2}$} (z);
\draw(z) --node[below]{$\frac{-i+j+k}{2}$} (y);
\draw(y) --node[left]{$\frac{i+j-k}{2}$}(x);
\end{tpic}
$$

This set of basis can be generalize to the surfaces with non-empty boundary. Let $V_{p, g, b; \vec{c}}$ be the vector space generated by the coloring of a fixed connected uni-trivalent graph $\Gamma_{g, b}$ satisfying the following three conditions.
\begin{enumerate}
\item[1)] The uni-trivalent graph $\Gamma_{g, b}$ is the retraction of $H_g$ and have $b$ univalent vertexes.
\item[2)] The colors of the edges connected to the univalent vertexes are $\vec{c}$.
\item[3)] At each trivalent vertexes, the coloring are admissible.
\end{enumerate}
For example, $V_{p, 4; 1, 1, 3, 3}$ is spaned by the following set of basis.

$$
V_{0, 4; 1, 1, 3, 3}=\Bigg \langle
\begin{tpic}[baseline=-.5ex, scale=0.7]
\draw[line width=1pt] (-0.707, 0.707) node[left]{1}--(0,0)--(-0.707, -0.707)node[left]{1};
\draw[line width=1pt] (0,0)--node[above]{2}(1,0);
\draw[line width=1pt] (1.707, 0.707) node[right]{3}--(1,0)--(1.707, -0.707)node[right]{3};
\end{tpic}
;
\begin{tpic}[baseline=-.5ex, scale=0.7]
\draw[line width=1pt] (-0.707, 0.707) node[left]{1} to[bend left=60] (-0.707, -0.707)node[left]{1};
\draw[line width=1pt] (1.407, 0.707) node[right]{3} to[bend right=60] (1.407, -0.707)node[right]{3};
\end{tpic}
\Bigg \rangle
$$

\begin{remark}The vector spaces we defined are closely related to the representation category of quantum group $SU(2)_{q}$. $(i, j, k)$ is an admissible triple is equivalent to that the vector space $Inv(V_i\otimes V_j\otimes V_k)$ is nonempty, where $V_i$ are Verlindre modules, e.g $V_{p, 4; 1, 1, 3, 3}\cong Inv(V_1\otimes V_1\otimes V_3 \otimes V_3)$ as vecter spaces.
\end{remark}
By the point of remark above, given a punctured surface $\Sigma_{g, b}$ and $p$, we say a coloring $\vec{c}$ is admissible if the dimension of $V_{p, g, b; \vec{c}}$ is positive.

\subsection{Mapping class group}
Let $S$ be the surface with or without boundary. The mapping class group of $S$ is defined to be
$$\M(S)=Diff(S; \partial S)/\{isotopy\}.$$
In this section, we define an action of $\M(\Sigma_{g, b})$ on vector space $V_{p, g, b, \vec{c}}$.

We first view $V_{p, g, b, \vec{c}}$ as vector space spanned by the skeins in handlebody $H_g$, with strands and clasps on the boundary if $b\neq 0$. Mapping class groups are generated by Dehn twists. We define the mapping class group action by defining actions of Dehn twists.

Here we introduce a new color $\Omega$, which is defined to be
$$\Omega=\sqrt{2/p} \sin(\pi/p)\Sigma_{i=0}^{p-2} (-1)^i [i+1] \phi_i,$$
where $\phi_i$ is the strand colored by $i$.

Consider a single closed curve $\gamma$ on $\Sigma_{g, b}$. The action of the Dehn twist $D_\gamma$ on the skein space is by adding a curve $\gamma$ colored by $\Omega$ with $(-1)$ framing along the boundary of $H_g$. This action turns out to be projective.

In particular, if $\gamma$ is the boundary of a disc perpendicular to a strand colored by $i$, then $D_{\gamma}v=(-1)^i A^{i(i+2)}v$ under some central extension where $v$ denotes the vector representing the skein.

The following theorem allows us to restrict the (projective)representations of $\M(S)$ to $\M(S'\subset S)$
\begin{theorem}
\label{inclusion}\cite{BD}[The inclusion homomorphism] Let S be a closed subsurface of a surface S'. Assume that S is not homeomorphic to a closed annulus and no component of S'-S is an open disk. Let $\eta : \M(S) \rightarrow \M(S')$ be the induced map. Let $\alpha_1, ...\alpha_m$ donote the boundary components of S that bound once-punctured disks in $S'-S$ and let $\{\beta_1, \gamma_1\},...,\{\beta_n, \gamma_n\}$ denote the pairs of boundary components of S that bound annuli in $S'-S$. Then the kernel of $\eta$ is the free abelian group:
$$ker(\eta)=<T_{\alpha_1},...,T_{\alpha_m},T_{\beta_1} T^{-1}_{\gamma_1},...,T_{\beta_n} T^{-1}_{\gamma_n}>$$
In particular, if no connected component of S'-S is an open annulus, an open disk, or an open once marked disk, then $\eta$ is injective.
\end{theorem}

\subsection{Notations and conventions}
$p$ is a fixed prime number throughout section 3 to section 6. The index $p$ in the vector space $V_{p, g, b; \vec{c}}$ will be hided for notational simplicity.

Let $[i]$ be the $i$-th quantum integer, i.e.
$$[i]=\frac{A^{2i}-A^{-2i}}{A^2-A^{-2}}$$

\section{A criterion for irreducibility}

\begin{lemma}
\label{cross1}\cite{Ku}Consider a Hermitian vector space $X$ over some field $F$ is a (projective) representation of some group $G$, and $G_1$, $G_2$ are two subgroups. Let $X=\bigoplus_I V_i$(resp. $\{X=\bigoplus_J W_{j}\}$) be an irreducible multiplicity free decomposition under the action of some central extension $\tilde{G_1}$(resp. $\tilde{G_2}$) of ${G_1}$(resp. $G_2$). Define a graph $C(X, G_1, G_2)$ on the set of irreducible summands, with an edge connecting $V_j$ and $W_k$ if there exists some element $v\in V_j$ and $w\in W_j$ such that $[v, w]\ne 0$.

If the graph is connected, then $X$ is irreducible as a representation of $G$.
\end{lemma}

Note that decompositions in Lemma 4 do not depend on the choice of central extensions. Thus, when proving irreducibility, $G$ and $\tilde{G}$ will not be distinguished. As discussed in \cite{R}, central extension does not affect the irreducibility.

With the setting of Lemma \ref{cross1}, each $W_j$(resp. $V_i$) is connected with some $V_i$(resp. $W_j$) since $W_j$(resp. $V_i$) is not $\{0\}$ as a set. Thus, to prove irreducibility, we just need to prove one side of this bipartite graph is connected. We formulate it in the following corollary.

\begin{corollary}
\label{cross}With the assumptions in Lemma 4, we can conclude irreducibility with one of the following conditions on the graph.
\begin{enumerate}
  \item[a)] for all $i, j \in I$, $V_i$, $V_j$ are connected by some path.
  \item[b)] there exist $i \in I$ such that $V_i$ is connected to $W_j$ for all $j \in J$.
\end{enumerate}
\end{corollary}

\section{Two base cases}

In this section, we discuss two base cases, where the surfaces are sphere with 4 punctures and one-holed torus.
\subsection{Sphere with 4 punctures}
\begin{lemma}$V_{0, 4; a, b, c, d}$ is irreducible under the action of $PB_4$ for all admissible $a, b, c, d$.  By admissible, we mean the vector space associated has positive dimension.
\end{lemma}

\begin{proof}
The Dehn twist around $\gamma_1$ and $\gamma_2$ gives 2 othogonal decompositions:
$$V_{0, 4; a, b, c, d}=\bigoplus _{i=\max \{|a-b|,|c-d|\}; i\equiv a-b\mod 2}^{\min \{a+b, c+d, 2p-4-a-b, 2p-4-c-d\}} V_i,$$
$$V_{0, 4; a, b, c, d}=\bigoplus _{j=\max \{|a-d|,|b-c|\}; i\equiv a-d\mod 2}^{\min \{a+d, c+b, 2p-4-a-d, 2p-4-c-b\}} W_j,$$
where $V_i$(resp. $W_j$)are spanned by the single vector $v_i$(resp. $w_j$).
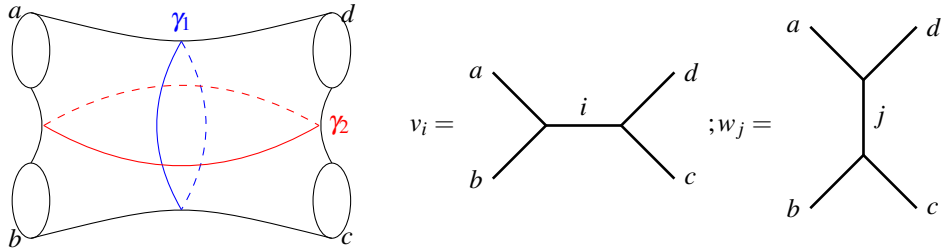
$$
\begin{tpic}[baseline=-.5ex]
\draw (-2, 1.5)node[left]{$a$}..controls (0, 1)..(2, 1.5)node[right]{$d$} (-2,-1.5)node[left]{$b$}..controls(0,-1)..(2, -1.5)node[right]{$c$};
\draw[blue](0, 1.12) to[bend right] (0, -1.12);
\draw[blue, dashed](0, 1.12) node[above]{$\gamma_1$} to[bend left] (0, -1.12);
\draw (-2, -0.5) to[bend right] (-2, 0.5) (2, -0.5) to[bend left](2, 0.5);
\draw[red](-1.83,0) to[bend right](1.83,0);
\draw[red, dashed](-1.83,0) to[bend left](1.83,0) node[right]{$\gamma_2$};
\draw(-2, 1) ellipse (0.25 and 0.5) (2, 1) ellipse (0.25 and 0.5) (-2, -1) ellipse (0.25 and 0.5) (2, -1) ellipse (0.25 and 0.5);
\end{tpic} \hspace{6mm}
v_i=
\begin{tpic}[baseline=-.5ex, line width=1pt]
\draw (-0.707, 0.707) node[left]{$a$}--(0,0)--(-0.707, -0.707)node[left]{$b$};
\draw (0,0)--node[above]{$i$}(1,0);
\draw (1.707, 0.707) node[right]{$d$}--(1,0)--(1.707, -0.707)node[right]{$c$};
\end{tpic}
; w_j=\begin{tpic}[baseline=2ex, line width=1pt]
\draw[rotate=90] (-0.707, 0.707) node[left]{$b$}--(0,0)--(-0.707, -0.707)node[right]{$c$};
\draw[rotate=90] (0,0)--node[right]{$j$}(1,0);
\draw[rotate=90] (1.707, 0.707) node[left]{$a$}--(1,0)--(1.707, -0.707)node[right]{$d$};
\end{tpic}
$$

When $p$ is an odd prime, $(-1)^{i}A^{i(i+2)}$ are different for different $i$. Thus, both decompositions are multiplicitiy free because they have different eigenvalues under the action of $D_{\gamma_t}$($t=1, 2$).

By lemma \ref{cross1}, the irreducibility of the representation of$PB_4$ is equivalent to the connectivity of the graph $C(V_{0, 4; a, b, c, d}, <D_{\gamma_1}>, <D_{\gamma_2}>)$. Before we proceed, we give the change of basis formula \cite{BHMV:kauffman1} for $\{v_i\}$ and $\{w_j\}$ below.
$$v_{i}=\sum_{j}
\left\{\begin{array}{ccc}
a & b & i \\
c & d & j 
\end{array}\right\} w_j=\sum_j\frac{
\begin{tpic}[line width=1pt, baseline=0ex]
\draw (0, 0) circle (1);
\draw (1, 0) node [left] {$j$};
\end{tpic}
\begin{tpic}[line width=1pt, baseline=0ex]
\draw (0, 0) circle (1);
\draw (0, -0.75) node{$j$};
\draw (0.85, 0.85) node{$a$};
\draw (-0.85, 0.85) node{$b$};
\draw (0, 0) to node[right]{$i$} (0, 1);
\draw (0, 0) to node[above]{$d$} (0.866, -0.5);
\draw (0, 0) to node[above]{$c$} (-0.866, -0.5);
\end{tpic}}{
\begin{tpic}[line width=1pt, baseline=0ex]
\draw (0, 0) circle (1);
\draw (-1, 0) to node[above]{$d$} (1, 0);
\draw (0.85, 0.85) node{$a$};
\draw (0.85, -0.85) node{$j$};
\end{tpic}  \begin{tpic}[line width=1pt, baseline=0ex]
\draw (0, 0) circle (1);
\draw (-1, 0) to node[above]{$b$} (1, 0);
\draw (0.85, 0.85) node{$c$};
\draw (0.85, -0.85) node{$j$};
\end{tpic} } w_j$$

where we follow the conventions in \cite{BHMV:kauffman1}:
$$<a, b, c>=
\begin{tpic}[line width=1pt, baseline=-0.5ex]
\draw (0, 0) circle (1);
\draw (-1, 0) to node[above]{$b$} (1, 0);
\draw (0.85, 0.85) node{$a$};
\draw (0.85, -0.85) node{$c$};
\end{tpic}; \hspace{4mm}
\left\langle\begin{array}{ccc}
a & b & i \\
c & d & j 
\end{array}\right\rangle
=\begin{tpic}[line width=1pt, baseline=-0.5ex]
\draw (0, 0) circle (1);
\draw (0, -0.75) node{$j$};
\draw (0.85, 0.85) node{$a$};
\draw (-0.85, 0.85) node{$b$};
\draw (0, 0) to node[right]{$i$} (0, 1);
\draw (0, 0) to node[above]{$d$} (0.866, -0.5);
\draw (0, 0) to node[above]{$c$} (-0.866, -0.5);
\end{tpic}$$
$V_i$ and $W_k$ are connected is equivalent to the span of $v_i$ in $\{w_j\}$ has non-zero coefficient on $w_k$. According to Corollary \ref{cross} case $(a)$, Our strategy is to to find some $i$ such that $V_i$ is connected to $W_j$ for all $j$.

Without loss of generality, we assume $a-b\ge|c-d|$(otherwise, we rotate the symbols), write $v_{a-b}$ as summation of $w_j$:
\begin{equation}
v_{a-b}=\Sigma_{j}
\left\{\begin{array}{ccc}
a & b & a-b \\
c & d & j 
\end{array}\right\} w_j
\end{equation}

If $j$ is admissible, $<b, c, j>$ and $<a, d, j>$ are nonzero, so we just need to check the tetrahedron symbols.
Theorem 2 of \cite{BHMV:kauffman1} gives an explicit formula for the tetrahedron symbol. As in \cite{BHMV:kauffman1}, Let $m_1=(a+b+i)/2$, $m_2=(a+d+j)/2$, $m_3=(b+c+j)/2$, $m_4=(i+d+c)/2$; $n_1=(a+b+c+d)/2$, $n_2=(b+i+d+j)/2$, $n_3=(a+i+c+j)/2$.

\begin{equation}
\left < 
                \begin{array}{ccc}
                 a & b & i \\
                 c& d  & j
                 \end{array} \right > =\sum_{z=\max{m_s}}^{\min{n_t}}\frac{\prod_{s, t}[n_s-m_t]!}{[a]![b]![c]![d]![i]![j]!} \frac{(-1)^z[z+1]!}{\prod_s[n_s-z]!\prod_t[z-m_t]!}.
\end{equation}
In general, the tetrahedron symbol is a summation over $z$ for all $\max_{t}{n_t}\le z\le \min_{s}{m_s}$. In our case, we only have one summand because $\max_{t}{n_t}= z= \min_{s}{m_s}$. which is
\begin{equation}
\left < 
                \begin{array}{ccc}
                 a & b & a-b \\
                 c& d  & j
                 \end{array} \right > =\frac{\prod_{s, t}[n_s-m_t]!}{[a]![b]![c]![d]![a-b]![j]!} \frac{(-1)^z[z+1]!}{\prod_s[n_s-z]!\prod_t[z-m_t]!}.
\end{equation}
Then we need to check the $q$-factorials in the above formula are less than $p$ to make sure all factors are nonzero. Note that $2(n_s-m_t)$ can be realized as summation of two labels of an admissible triple subtracting the other one, which is always less than $2(p-2)$. Thus $n_s-m_t\le p-2$, which implies $[n_s-m_t]!\ne 0$; $z$ is half of the sum of labels of an admissible triple, so $z\le p-2$, $[z+1]!\ne 0$.

The above computation showed that $V_{a-b}$ are connected to $W_j$ for all $j$. By Corollary \ref{cross}(a), $V_{0, 4;a, b, c, d}$ is irreducible under the action of $PB_4$.

\end{proof}

\subsection{One-holed torus}

This case has been studied by G. Patrick and G. Masbaum \cite{GM11}. They proved irreducibiliy of $V_{g, 1; 2a}$ for any $g$ when $p$ is an odd prime. In our case, we just need $g=1$ to start the induction. For completeness, we put a more elementary proof here, and we would like to thank Julien Korinman \cite{Kor} for teaching us the proof.

\begin{lemma}
\label{one}Let $p$ be an odd prime, $1\le a\le \frac{p-3}{2}$, $V_{1, 1; 2a}$ is irreducible under the action of $\M(\Sigma_{1,1})$.
\end{lemma}
\begin{proof}
In \cite{GG}, their computation showed that the Hopf pairing (see figure below) $<,>$ of $V_{1, 1; 2a}$ is a nondegenerate bilinear form.
$$
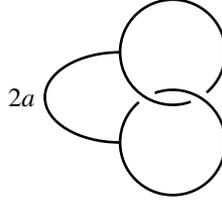
\begin{tpic}[line width=1pt]
\draw(-1,-0.6)node[left]{$2a$} arc (180:90:1 and 0.6) (-1,-0.6) arc (180:270:1 and 0.6);
\draw(1.4,0)node[right]{} arc (0:290:0.7) (1.4,0) arc (0:-50:0.7);
\draw(1.4, -1.2)node[right]{} arc (0:110:0.7) (1.4, -1.2) arc (0:-230:0.7);
\end{tpic}$$

 Let ${v_j}$ be the lolipop basis shown below, and $w_i=D_{\beta}^i v_0$ be the vector derived from applying Dehn twist along $\beta$ $i$ times to $v_0$.\\
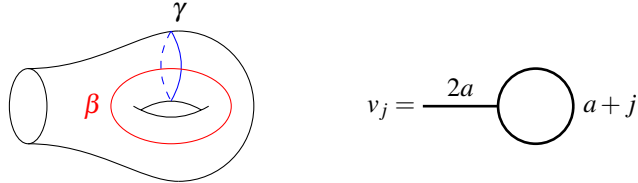
$$
\begin{tpic}[baseline=-.5ex]
\draw(0, 1) node[above]{$\gamma$} arc (90:-90:1);
\draw(-2,0.5) ..controls (-1.1,0.5) and (-0.5,1).. (0, 1);
\draw(-2,-0.5) ..controls (-1.1,-0.5) and (-0.5,-1).. (0, -1);
\draw (-2, 0) ellipse (0.25 and 0.5);
\draw (-0.6, 0) to[bend right] (0.4, 0) (-0.5, -0.05) to[bend left] (0.3, -0.05);
\draw[blue] (-0.1,1) to[bend left] (-0.1,0.07);
\draw[dashed, blue] (-0.1,1) to[bend right] (-0.1,0.07);
\draw[red] (-0.9, 0)node[left]{$\beta$} arc (180:-180:0.8 and 0.5);
\end{tpic} \hspace{15mm}
v_j=
\begin{tpic}[line width=1pt, baseline=-.5ex]
\draw(0,0)--node[above]{$2a$}(1,0);
\draw(2,0)node[right]{$a+j$} arc (0:360:0.5);
\end{tpic}$$
We compute the Hopf pairing of $w_i$ and $v_j$, where $0\le i, j\le p-a-2$:
$$<w_i, v_j>=<D_{\beta}^i v_0, v_j>=<v_0, D_\gamma^i v_j>=(-1)^{ij}A^{i(j+a)(2+j+a)}<v_0, v_j>$$
When $p$ is an odd prime, $(-1)^j A^{(j+a)(2+j+a)}$ are different complex numbers, so matrix $V=\{(-1)^{ij}A^{i(j+a)(2+j+a)}\}_{ij}$ is a Vandermonde matrix, thus invertible. $\{<w_i, v_j>\}_{ij}=V\cdot diag\{<v_0, v_j>\}$. $<v_0, v_j>\ne 0$(c. f. \cite{GG} Page 100), so the product of the two matrices is invertible, which implies ${w_i}$ spans $V_{1, 1; 2a}$.

Consider the bipartite graph $C(V_{1, 1; 2a}, <D_\beta>, <D_\gamma>)$. $<v_0>$ is invariant under the action of $D_\gamma$. The decomposition of $V_{1, 1; 2a}$ under the action of $<D_\gamma>$ is multiplicity free. The argument above showed that $v_0$ have component in all eigenspaces of $D_\beta$ and each eigenspaces are 1-dimensional. That is, $<v_0>$ is connected to all eigenspaces of $<D_\beta>$, and all the eigenspaces are multiplicity free due to eigenvalue test. By corollary \ref{cross}(a), $V_{1, 1; 2a}$ is irreducible.

\end{proof}

\section{The induction}
In this section, we will develope the induction steps. The idea is the following: We can decompose the representation by restricting it to a mapping class group of a subsurface $S'\subset S$. Usually $S-S'$ is a cylinder $\alpha \times I$. The decomposition depends on the choice of $\alpha$ on $S$. Given two different such decompositions, we study the connectivity of the bipartite graph described in Lemma \ref{cross1} to conclude the irreducibility.

The following three lemmas provide us the tools for the induction on genus $g$ and the number of boundaries $b$.

\begin{lemma}
\label{1}
Let $p$ be an odd prime and $g\ge 1$, Suppose $V_{g, 1, c}$ and $V_{g-1, 2; i, j}$ are irreducible as $\M(\Sigma_{g, 1})$ and $\M(\Sigma_{g-1, 2})$ representation respectively for all admissible $c, i, j$, then $V_{g, 2; a, b}$ is an irreducible representation of $\M(\Sigma_{g, 2})$ for all $a, b$.
\end{lemma}
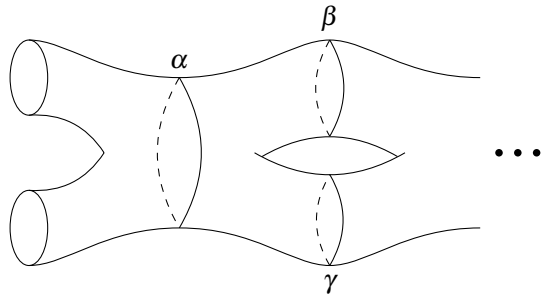
$$
\begin{tpic}

\draw(-4,-1.5)
    .. controls (-3.5, -1.5) and (-3,-1) .. (-2,-1)
    .. controls (-1,-1) and (-0.5,-1.5) .. (0,-1.5)
..controls (0.5, -1.5) and (1, -1)..(2, -1);

\draw(-4,1.5)
    .. controls (-3.5, 1.5) and (-3,1) .. (-2,1)
    .. controls (-1,1) and (-0.5,1.5) .. (0,1.5)
..controls (0.5, 1.5) and (1, 1)..(2, 1);

\draw(-4, 0.5) to[bend left] (-3, 0) (-4, -0.5) to[bend right] (-3, 0);
\draw(-1, 0) to[bend right] (1, 0);
\draw(-0.9,-0.05) to[bend left](0.9,-0.05);
\draw (-4, 1) ellipse (0.25 and 0.5);
\draw (-4, -1) ellipse (0.25 and 0.5);
\draw (-2, 1) to[bend left] (-2, -1);
\draw[dashed] (-2, 1) node[above]{$\alpha$} to[bend right] (-2, -1);
\draw (0, 1.5) to[bend left] (0, 0.22);
\draw[dashed] (0, 1.5) node[above]{$\beta$} to[bend right] (0, 0.22);
\draw (0, -1.5) to[bend right] (0, -0.29);
\draw[dashed] (0, -1.5) node[below]{$\gamma$} to[bend left] (0, -0.29);
\fill (2.25, 0)circle (0.05);
\fill (2.5,0)circle (0.05);
\fill (2.75, 0)circle (0.05);
\end{tpic}$$
\begin{proof}Let $S$ be a representative of surface $\Sigma_{g, b}$ as above. We restrict $V_{g, 2; a, b}$ to the action of the image of $\M(S-\alpha\times I)\rightarrow \M(S)$ and the image of $\M(S-(\beta\cup\gamma)\times I)\rightarrow\M(S)$ respectively. According to Theorem 3, the kernels factor though the corresponding representations. Notice $\M(S-\alpha\times I) \cong \M(\Sigma_{0, 3})\times \M(\Sigma_{g, 1})$ and $\M(S-(\beta\cup\gamma)\times I)\cong \M(\Sigma_{0, 4})\times\M(\Sigma_{g-1, 2})$. Thus, with the assumptions in this lemma, $V_{g, 2; a, b}$ has the following two irreducible decompositions accordingly.

$$V_{g, 2; a, b}=\bigoplus_c A_c=\bigoplus_c V_{g, 1; c}\otimes V_{0,3; c, a, b}; V_{g, 2; a, b}=\bigoplus_{i, j} B_{i, j}=\bigoplus_{i, j} V_{g-1, 2; i, j}\otimes V_{0, 4; i, j, a, b},$$

where $A_c$ and $B_{i, j}$ are invariant spaces of $D_{\alpha}$ and $D_{\beta} \times D_{\gamma}$ respectively. They are spanned by the graphs shown below.\\
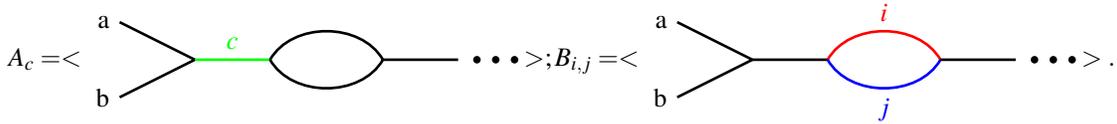
$$
A_c=<
\begin{tpic}[line width=1pt, baseline=-.5ex]
\draw(-1, 0.5) node[left]{a} to (0,0) to (-1, -0.5) node[left]{b};
\draw[line width=1pt, green] (0, 0) to node[above]{$c$} (1, 0);
\draw(1, 0) to[bend left=60](2.5, 0);
\draw(1, 0) to[bend right=60](2.5, 0);
\draw(2.5, 0) to (3.5, 0);
\fill (3.75, 0)circle (0.05);
\fill (4, 0)circle (0.05);
\fill (4.25, 0)circle (0.05);
\end{tpic}
>
;
B_{i, j}=<
\begin{tpic}[line width=1pt, baseline=-.5ex]
\draw(-1, 0.5) node[left]{a} to (0,0) to (-1, -0.5) node[left]{b};
\draw (0, 0) to (1, 0);
\draw[line width=1pt, red](1, 0) to[bend left=60] node[above]{$i$} (2.5, 0);
\draw[line width=1pt, blue](1, 0) to[bend right=60] node[below]{$j$} (2.5, 0);
\draw(2.5, 0) to (3.5, 0);
\fill (3.75, 0)circle (0.05);
\fill (4, 0)circle (0.05);
\fill (4.25, 0)circle (0.05);
\end{tpic}
>
.$$
When $p$ is odd prime, these invariant spaces have different eigenvalues for the Dehn twists, so the decompositions are multiplicity free for all central extensions.

According to Lemma \ref{cross1}, we just need to prove the bipartite graph $C(V_{g, 2; a, b}, \M(S-\alpha\times I), \M(S-(\beta\cup\gamma)\times I))$ is connected.

Note that if $(c, i, j)$ is an admissible triple, $A_c$ and $B_{i, j}$ are connected because the following element will be in the intersection of $A_c$ and $B_{i, j}$.

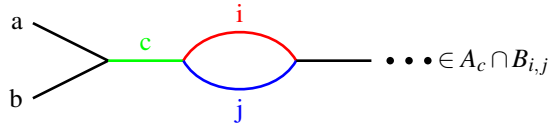
$$
\begin{tpic}[line width=1pt, baseline=-.5ex]
\draw(-1, 0.5) node[left]{a} to (0,0) to (-1, -0.5) node[left]{b};
\draw[line width=1pt, green] (0, 0) to node[above]{c} (1, 0);
\draw[line width=1pt, red](1, 0) to[bend left=60] node[above]{i} (2.5, 0);
\draw[line width=1pt, blue](1, 0) to[bend right=60] node[below]{j} (2.5, 0);
\draw(2.5, 0) to (3.5, 0);
\fill (3.75, 0)circle (0.05);
\fill (4, 0)circle (0.05);
\fill (4.25, 0)circle (0.05);
\end{tpic}
\in A_c \cap B_{i, j}
$$
Then notice for all $c$, $(c, \frac{p-3}{2}, \frac{p-3}{2})$ will be an admissible triple, so all $A_c$ are connected to $B_{\frac{p-3}{2}, \frac{p-3}{2}}$. $B_{\frac{p-3}{2}, \frac{p-3}{2}}$ is of positive dimension for all $g$. Thus, $V_{g, 2; a, b}$ is an irreducible $\M(\Sigma_{g, 2})$ representation by corollary \ref{cross}.
\end{proof}

If all boundaries are colored by $0$, we should consider it as the closed surface. This case not only shows up in the question itself, but also contribute to the induction. Although closed surface case have been proved by J. Roberts in \cite{R}, we still put this case in our induction for completeness.

\begin{lemma}
\label{2}Let $p$ be an odd prime and $g\ge 2$. Assume $V_{g-1, 2; i, i}$ is irreducible under the action of $\M(\Sigma_{g-1, 2})$ for all $i$. Then $V_{g, 0}$ is irreducible under the action of $\M(\Sigma_{g})$.
\end{lemma}
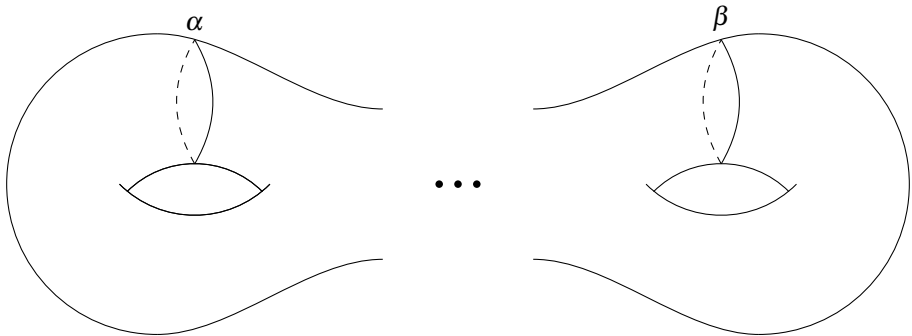
$$
\begin{tpic}
\draw(-1, -1) ..controls (-2, -1) and (-3, -2).. (-4, -2);
\draw (-4, 2) arc (90: 270: 2);
\draw(-1, 1) ..controls (-2, 1) and (-3, 2).. (-4, 2);
\fill (-0.25, 0)circle (0.05);
\fill (0, 0)circle (0.05);
\fill (0.25, 0)circle (0.05);
\draw(1, -1) ..controls (2, -1) and (3, -2).. (4, -2);
\draw (4, -2) arc (-90: 90: 2);
\draw(1, 1) ..controls (2, 1) and (3, 2).. (4, 2);
\draw(-4.5, 0) to[bend right=45] (-2.5, 0);
\draw(-4.4, -0.1) to[bend left=45] (-2.6, -0.1);

\draw(-4.5, 0) to[bend right=45] (-2.5, 0);
\draw(-4.4, -0.1) to[bend left=45] (-2.6, -0.1);

\draw(4.5, 0) to[bend left=45] (2.5, 0);
\draw(4.4, -0.1) to[bend right=45] (2.6, -0.1);

\draw (-3.5, 0.27) to[bend right] (-3.5, 1.92) node[above]{$\alpha$};
\draw[dashed] (-3.5, 0.27) to[bend left] (-3.5, 1.92);

\draw (3.5, 0.27) to[bend right] (3.5, 1.92) node[above]{$\beta$};
\draw[dashed] (3.5, 0.27) to[bend left] (3.5, 1.92);
\end{tpic}$$
\begin{proof}Let $S$ be a representative of $\Sigma_{g, 0}$ as above. We restrict $V_{g, 0}$ to the representation of subgroups of the image of $\M(S-\alpha\times I)$ and the image of $\M(S-\beta\times I)$ in $\M(S)$ respectively. According to Theorem \ref{inclusion}, the kernels factor though the corresponding representations. Notice that both $\M(S-\alpha\times I)$ and $\M(S-\beta\times I)$ are isomorphic to $\M(\Sigma_{g-1, 2})$. Thus, with the assumptions in this lemma, $V_{g, 0}$ has the following two irreducible decompositions accordingly.

$$V_{g, 0}\cong \oplus_i A_i=V_{g-1, 2; i, i}; V_{g, 0}\cong \oplus_j B_j=V_{g-1, 2; j, j}$$

Where $A_i$ and $B_j$ are invariant spaces of $D_{\alpha}$ and $D_{\beta}$ respectively. They are spanned by the graphs shown below:\\
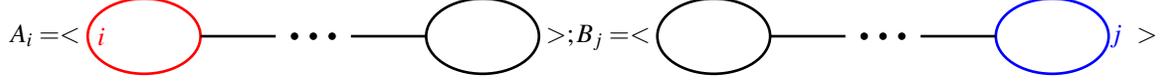
$$
A_i=<
\begin{tpic}[line width=1pt, baseline=-.5ex]
\draw[line width=1pt, red] (-2.25, 0) ellipse (0.75 and 0.5);
\draw (-1.5, 0) to (-0.5, 0);
\fill (-0.25, 0)circle (0.05);
\fill (0, 0)circle (0.05);
\fill (0.25, 0)circle (0.05);
\draw (2.25, 0) ellipse (0.75 and 0.5);
\draw (1.5, 0) to (0.5, 0);
\draw[red] (-2.82, 0)node {$i$};
\end{tpic}
>
;
B_j=<
\begin{tpic}[line width=1pt, baseline=-.5ex]
\draw (-2.25, 0) ellipse (0.75 and 0.5);
\draw (-1.5, 0) to (-0.5, 0);
\fill (-0.25, 0)circle (0.05);
\fill (0, 0)circle (0.05);
\fill (0.25, 0)circle (0.05);
\draw[line width=1pt, blue] (2.25, 0) ellipse (0.75 and 0.5);
\draw (1.5, 0) to (0.5, 0);
\draw[blue] (3.1, 0)node {$j$};
\end{tpic}
>
$$.

When $p$ is odd prime, these invariant spaces have different eigenvalues under the action of $D_\alpha$ and $D_\beta$ respectively, so the decompositions are multiplicity free for all central extension of $\M(\Sigma_{g})$. Notice the following element is in the intersection of $A_i$ and $B_j$:
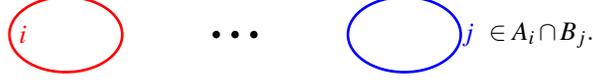
$$
\begin{tpic}[baseline=-.5ex]
\draw[line width=1pt, red] (-2.25, 0) ellipse (0.75 and 0.5);

\fill (-0.25, 0)circle (0.05);
\fill (0, 0)circle (0.05);
\fill (0.25, 0)circle (0.05);
\draw[line width=1pt, blue] (2.25, 0) ellipse (0.75 and 0.5);
\draw[line width=1pt, red] (-2.82, 0)node {$i$};
\draw[line width=1pt, blue] (3.1, 0)node {$j$};
\end{tpic}
\in A_i\cap B_j.
$$
Argument above showed that the bipartite graph $<V_{g, 0}; \M(S-\alpha\times I), \M(S-\beta\times I)>$ is complete. Thus, the representation is irreducible by Lemma \ref{cross1}.
\end{proof}

Before we introduce the next lemma, we define $\prec$ to be the lexicographical order on pair $(g, b)$.

\begin{lemma}
\label{3}Let $p$ be an odd prime and $(g, b)\notin \{(0, 1), (0, 2), (0, 3), (0, 4), (1, 1), (1, 2)\}$ and $b\ge 0$. Consider vector space $V_{g, b; \vec{c}}$ with $b\ge 1$ and one of the boundary is colored by $a\ne p-2$. If $V_{g', b'; \vec{c'}}$ is irreducible under the action of $\M(\Sigma_{g', b'})$ for all $(g', b')\prec (g, b)$ and any color $\vec{c'}$, $V_{g, b; \vec{c}}$ is irreducible under the action of $\M(\Sigma_{g, b})$.

\end{lemma}
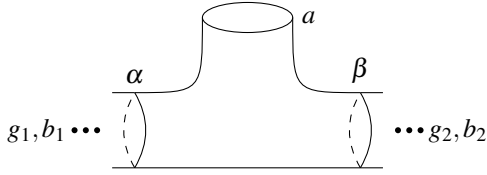
\begin{tpic}
\draw (-1.8, 0) to (1.8, 0);
\draw (-1.5, 0) to[bend right] (-1.5, 1)node[above]{$\alpha$};
\draw[dashed] (-1.5, 0) to[bend left] (-1.5, 1);
\draw (1.5, 0) to[bend right]  (1.5, 1)node[above]{$\beta$};
\draw[dashed] (1.5, 0) to[bend left] (1.5, 1);
\draw (0, 2) ellipse (0.6 and 0.2);
\draw (1.5, 1) ..controls (0.6, 1) and (0.6, 1).. (0.6, 2) node[right]{$a$};
\draw (-1.5, 1) ..controls (-0.6, 1) and (-0.6, 1).. (-0.6, 2);
\draw (-1.8, 1) to (-1.5, 1) (1.5, 1) to (1.8, 1);

\fill (-2, 0.5)circle (0.04);
\fill (-2.15, 0.5)circle (0.04);
\fill (-2.3, 0.5)circle (0.04);
\fill (2, 0.5)circle (0.04);
\fill (2.15, 0.5)circle (0.04);
\fill (2.3, 0.5)circle (0.04);
\draw (-2.8, 0.5) node {$g_1, b_1$};
\draw (2.8, 0.5) node {$g_2, b_2$};
\end{tpic}
\begin{proof}
With the numerical conditions on $(g, b)$ described in the lemma, we can find $(g_1, b_1)\prec(g, b)$ and $(g_2, b_2)\prec(g, b)$ satisfying $g=g_1+g_2$ and $b=b_1+b_2-1$.
By restricting to the group actions of $\M(S-\alpha\times I)$ and $\M(S-\beta\times I)$, we can decompose the representation $V_{g, b, \vec{c}}=V_{g_1+g_2, b_1+b_2+1; \vec{c_1}, \vec{c_2}, a}$ in following two ways:\\
$$A_{i}=V_{{g_1}, b_1+1; \vec{c_1}, i}\otimes V_{{g_2}, b_2+2; \vec{c_i}, i, a}$$
$$B_{j}=V_{{g_1}, b_1+2; \vec{c_1} a, j}\otimes V_{{g_2}, b_2+1; \vec{c_2}, j}$$

Where $A_i$ and $B_j$ are invariant spaces of $D_{\alpha}$ and $D_{\beta}$ respectively. They are spanned by the graphs shown below:\\
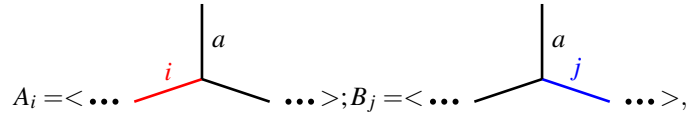
$$
A_i=<
\begin{tpic}[line width=1pt] 
\draw[line width=1pt, red](-0.9, 0) to node[above]{$i$} (0, 0.3);
\draw(0.9, 0) to (0, 0.3);
\draw(0, 0.3) to node[right]{$a$} (0, 1.3);
\fill (-1.45, 0)circle (0.04);
\fill (-1.3, 0)circle (0.04);
\fill (-1.15, 0)circle (0.04);
\fill (1.45, 0)circle (0.04);
\fill (1.3, 0)circle (0.04);
\fill (1.15, 0)circle (0.04);
\end{tpic}
>; B_j=<
\begin{tpic}[line width=1pt]
\draw(-0.9, 0) to (0, 0.3);
\draw[line width=1pt, blue](0.9, 0) to node[above]{$j$} (0, 0.3);
\draw(0, 0.3) to node[right]{$a$} (0, 1.3);
\fill (-1.45, 0)circle (0.04);
\fill (-1.3, 0)circle (0.04);
\fill (-1.15, 0)circle (0.04);
\fill (1.45, 0)circle (0.04);
\fill (1.3, 0)circle (0.04);
\fill (1.15, 0)circle (0.04);
\end{tpic}
>,
$$
where $i$(resp. $j$) run through all colors such that both of their tensor factor have positive dimension. That is, if both $g_1$ and $g_2$ are positive, $i$ and $j$ run through all colors satisfying the parity condition. If $g_1$ or $g_2$ is zero, then we have extra inequality constraints for $i$ and $j$, but these constraints always reduce to an interval.

By the assumption, $V=\oplus_i A_i$(resp.$V=\oplus_j B_j$) is an irreducible and multiplicity free(by checking eigenvalues of the Dehn twist around the circle in the graph) decomposition under the group action of $\M(\Sigma_{g_1, b_1+1})\times \M(\Sigma_{g_2, b_2+2})$(resp. $\M(\Sigma_{g_1, b_1+2})\times \M(\Sigma_{g_2, b_2+1})$), and we noticed that if $(i, j, a)$ is an admissible triple,  $A_i$ and $B_j$ are connected because the following element is in the intersection:
$$\begin{tpic}[line width=1pt]
\draw[line width=1pt, red](-0.9, 0) to node[above]{$i$} (0, 0.3);
\draw[line width=1pt, blue](0.9, 0) to node[above]{$j$} (0, 0.3);
\draw(0, 0.3) to node[right]{$a$} (0, 1.3);
\fill (-1.45, 0)circle (0.04);
\fill (-1.3, 0)circle (0.04);
\fill (-1.15, 0)circle (0.04);
\fill (1.45, 0)circle (0.04);
\fill (1.3, 0)circle (0.04);
\fill (1.15, 0)circle (0.04);
\end{tpic}
\in A_i\cap B_j$$
Now we proves the following claim: Suppose $i\ge i'$ and there exist $j, j'$ such that $(i, j, a)$ and $(i', j', a)$ are admissible triples, then there exist some $j''$ such that $(i, j'', a)$ and $(i-2, j'', a)$ are admissible triples.

If $(i-2, j, a)$ is admissible, we are done. Otherwise, one of the inequality conditions must fail, so we have either $i+j=a$ or $j-i=a$:
\begin{enumerate}
\item[Case 1] $i+j=a$: $i'+j'\ge a$ and $i'< i$ implies $j'>j$, so $B_{j+2}$ is not $\{0\}$. Let $j''=j+2$, we check $(i-2)+j''=a$, $|(i-2)-j''|=|i-4+j|\le|i-4|+j$. Since $i> i'$, $i\ge 2$, so $|i-4|+j\le i+j=a$. Thus $(i-2, j+2, a)$ is an admissible triple.
\item[Case 2] $j-i=a$: $j'-i'\le a$ and $i'< i$ implies $j'<j$, so $B_{j-2}$ is not $\{0\}$. Let $j''=j-2$, we check $|(i-2)-j''|=a$, $(i-2)+j''=j-i+2i-4\ge j-i=a$. Thus $(i-2, j-2, a)$ is an admissible triple.

\end{enumerate}

We take $i'$ to be the least $i$ such that $A_i\ne{0}$. The claim said $B_{j''}$ is connected to $A_i$ and $A_{i-2}$. This gives us paths connect all pairs of adjacent $A_i's$. Together with corollary \ref{cross}(b), we proved $V_{g_1+g_2, b_1+b_2+1; \vec{c_1}, \vec{c_2}, a}$ is irreducible.

\end{proof}

\section{Proof of the theorem 1.1}
In this section, we prove the theorem 1.1.

\begin{proof} We are going to do induction on the lexicographical order $\prec$ of pairs (g, b). We just need to show that for all $V_{g, b;\vec{c}}$ where $(g, b)\notin \{(0, 1), (0, 2), (0, 3), (0, 4), (1, 1), (1, 2)\}$, The irreducibility of $V_{g, b; \vec{c}}$ is implied by some $V_{g', b'; \vec{c'}}$ satisfying $(g', b')\prec (g, b)$.

We first consider the case that one of the boundary is colored by $p-2$. Because summation on colors on the boundary should be an even number, so $b\ge 2$. Pick another boundary that colored by $i$. We have the following isomorphism of $\M(\Sigma_{g, b-1})$ representations:
$$V_{g, b; \vec{c}}=V_{0, 3; p-2, i, p-2-i}\otimes V_{g, b-1; \vec{c}, p-2-i \setminus\{p-2, i\}}$$
By this, we lower the number of boundary components by $1$.

If $b\ge 1$ and none of the boundary is colored by $p-2$, and $(g, b)\notin \{(0, 1), (0, 2), (0, 3), (0, 4), (1, 1), (1, 2)\}$, Lemma \ref{3} implies it is enough to show $V_{g', b';\vec{c'}}$ are irreducible for all $(g', b')\prec(g, b)$.

If $b=0$ and $g\le 2$, Lemma \ref{2} implies it is enough to show $V_{g, 2; i, i}$ is irreducible.

If $(g, b)=(1, 2)$, Lemma \ref{1} implies it is enough to show $V_{1, 1; a}$ is irreducible.

For Base cases:
\begin{itemize}
\item $(g, b)=(0, 1), (0, 2), (0, 3)$. The vector spaces are either 0 or 1-dimensional.
\item $(g, b)=(1, 0)$. This is Weil representation, and it is irreducible.
\item $(g, b)=(0, 4)$ and $(1, 1)$. We proved them in section 4.
\end{itemize}
The above induction proves $V_{g, b; \vec{c}}$ is irreducible for all $g, b$ and all coloring $\vec{c}$.

\end{proof}

\section{Tools for prove denseness}
In the section, we introduce some tools to study images of group
homomorphisms that originate with Goursat's Lemma.

In generic $q$ case, we no longer have the Hermitian form on the skein spaces. Thus, we rephrase lemma \ref{cross1}.
\begin{lemma}
\label{cross2}\cite{Ku}Consider a vector space $X$ over some field $F$ is a (projective) representation of some group $G$, and $G_1$, $G_2$ are two subgroups. Let $X=\bigoplus_I V_i$(resp. $\{X=\bigoplus_J W_{j}\}$) be an irreducible multiplicity free decomposition under the action of some central extension $\tilde{G_1}$(resp. $\tilde{G_2}$) of ${G_1}$(resp. $G_2$). Define a directed graph $C(X, G_1, G_2)$ on the set of irreducible summands, with an edge from $V_j$ to $W_k$ if there exists some element $v\in V_j$ has nonzero component in $W_k$.

If the graph is strongly connected, then $X$ is irreducible as a representation of $G$.
\end{lemma}

The following theorem is often called non-commutative Chinese remainder theorem, that locally surjectivity implies global surjectivity.

\begin{theorem}Suppose that each of $G_1,G_2,...,G_l$ is a minimal simple Lie group or a non-abelian finite simple group, and suppose that
$$H \subset G = G_1 \times G_2 \times ... \times G_l$$
is a closed subgroup that surjects onto each factor $G_k$. Then
H is a diagonal subgroup of $G$.
\end{theorem}
In our case, group $H$ is the Zariski closure of the mapping class group and $G_i$ are complex simple Lie groups.
Thus, we have a corollary for our case.

\begin{corollary}
\label{nccrt}If $$f: H\rightarrow PSL(W_1)\times PSL(W_2)\times...\times PSL(W_n)$$
is surjective when restrict to $H\rightarrow G_i$ for all $i$, and non of pairs $W_i$ and $W_j$ are isomorphic or dual to each other, then $f$ is surjective.
\end{corollary}

Lastly, we need the following surjectivity theorem:

\begin{theorem}
\label{surj}
Let $V$ be a finite-dimensional complex representation of a connected Lie group $G$, Let $H\subset G$ be a closed, connected subgroup, and let 
$$V|_H\cong \bigoplus_{k=1}^{n}W_k$$
be the decomposition of the restricted representation. Suppose that:
\begin{enumerate}
\item $V$ is $G$-irreducible.
\item At most one of $W_i$ is of dimension 2 and at most one of $W_i$ is of dimension 1.
\item For every $j\ne k$, the summands $W_j$ and $W_k$ are neither isomorphic nor dual as projective representation of $H$.
\item For each j, $H$ surjects onto $PSL(W_j)$.
\end{enumerate}

Then $G$ surjects onto $PSL(V)$.

\end{theorem}
\begin{proof}
Let $L$ be the Lie algebra of $G$, $K\subset L$ be the corresponding Lie algebra of $H$. Then we have the splitting
$$\gl(V)=\Sl(V)\oplus \C$$.

The non-commutative Chinese remainder theorem for Lie algebra gives us that $H$ is jointly surjective:
$$H\twoheadrightarrow \bigoplus_k \Sl(W_k)$$.

Meanwhile we have the partial decomposition
$$\Sl(V)=\bigoplus_{j\ne k}W_j\otimes W_{k}^{*}\oplus\bigoplus_{k}\Sl(W_k)\oplus \C ^{n-1} $$.

Notice that if for some $i$, $\dim W_i > 2$, then $W_i\ncong W_i^*$ as $\Sl(W_i)$ representation. Together with the 3rd assumption in the theorem, $W_i\otimes W_j^*$ and $W_j\otimes W_i^*$ are both unique in the decomposition in the sense that they are not isomorphic to each other and all other summands as a representation of $H$. That is, The off-diagonal blocks are multiplicity free if and only if at most one of $W_i$ satisfying $\dim W_i\le 2$. In the rest of the proof, we discuss it case by case. Without lose of generality, we assume the $\dim W_i\le \dim W_j$ if $i\le j$.
\begin{enumerate}
    \item [Case 1] $\dim W_2\ge 3$.
    
    As we discussed above, the off diagonal summands are multiplicity free in this case. so the image of $L$ in $\Sl(V)$ contains a subset of the off-diagonal blocks $W_j\otimes W_k^*$ and some subspace of $\C^{n-1}$.
    We can make a directed graph $\Gamma$ with a directed edge $k\rightarrow j$ for every off-diagonal block $V_j\otimes V_k^*$ which is in the image of $L$. We claim $\Gamma$ is complete.
    
    First, we prove $\Gamma$ is strongly connected. Assuming otherwise, $\Gamma$ would have a strongly connected component $C$ with no outward edges. In this case $\bigoplus_{k\in C} W_k$ would be a non-trivial subrepresentation of $L$. $G$ is connected, so $\bigoplus_{k\in C} W_k$ is also a subrepresentation of $G$, contradicting the hypothesis that $V$ is irreducible.
    
    Second, we prove that $\Gamma$ is transitively closed. Suppose that $i\rightarrow k\rightarrow j$ is a path of length two using three distinct vertices. Choose two operators
    $$Y\in \Hom (W_l, W_k)\cong W_k\otimes W_l^*$$
    $$X\in \Hom (W_k, W_j)\cong W_k\otimes W_l^*$$
    whose product
    $$XY\in \Hom(W_l, W_j)\cong W_j\otimes W_l^*$$
    is non-zero. Then $[X, Y]$ is a non-zero element in $W_j\otimes W_l^*$. The image of $L$ thus contains some elements of $W_j\otimes W_l^*$, and therefore it contains all of them and $\Gamma$ is transitively closed.
    
    If $\Gamma$ is both strongly connected and transitively close, then it is the complete directed graph. In the final step, choose some basis of $V$ that refines the decomposition of $V|_H$. In this basis, the image of $L$ contains all off-diagonal elementary matrices, so there commutators gives us all of $\Sl(V)$
    
    \item [Case 2] $n\ge 3$ and $\dim W_2=2$, $\dim W_1=1$.
    
    Define $V'=\bigoplus_{k=2}^{n}W_k$. The identical argument as in case 1 shows that $\Sl(V')$ is in the image of $L$. Thus we have decomposition
    $$\Sl(V)=\Sl(V')\oplus V' \oplus V'^*\oplus \C$$
    as representation of $\Sl(V')$. The decomposition is multiplicity free since $\dim V' \ge 3$. $V$ is irreducible as $L$ representation, so The image of $L$ contains both $V$ and $V^*$. The commutators will give us all of $\Sl(V)$.
    
    \item [Case 3] $n=2$ and $\dim W_2=2$, $\dim W_1=1$.
    
    In this case, the $V$ is a $3$-dimensional representation. we list all the Lie subalgebra of $\Sl(V)$. Only $\Sl(V)$ itself makes $V$ irreducible.
    
\end{enumerate}
\end{proof}
Before we get to the last section, we introduce a lemma conjectured by the first author and supported by David Speyer. It is useful for analyzing weight diagrams.

\begin{lemma}[Speyer] Let $D(\lambda)$ be the set of weights in the weight diagram of the irreducible representation $V(\lambda)$ of a simple Lie algebra $\g$. Let $S(\lambda)=\lambda-D(\lambda)$ as a set of vectors in $\h^*$. The indecomposable elements of $S(\lambda)$ are roots of $\g$.
\end{lemma}
\begin{proof}
Let $R$ be the set of positive roots of $\g$. $C=S(\lambda)\cap R$, $C'=R-C$ and $v_{\lambda}$ be a vector in $V(\lambda)$ have weight $\lambda$. Notice $U(\mathfrak{n}_-)(v_\lambda)= V(\lambda)$. By PBW theorem, $U(\mathfrak{n}_-)$ is spanned by monomials of form $f_{c_1} f_{c_2}...f_{c_m} f_{c'_1} f_{c'_2}...f_{c'_n}v_{\lambda}$ where $c_i\in C$ and $c'_j\in C'$. Since $\lambda -c'_j \notin D(\lambda)$, we have $f_{c'_j}v_\lambda=0$ if $c'_j\ne 0$, so $V(\lambda)$ is spanned by $f_{c_1} f_{c_2}...f_{c_m} v_\lambda$. That is, for all weight $\lambda'\in D(\lambda)$ such that $\lambda-\lambda'$ indecomposable in $S(\lambda)$. $\lambda'$ is either $\lambda$ or $v_{\lambda'}=f_{c_i} v_{\lambda}$. This proves $\lambda-\lambda_i=c_i$ is a positive root.
\end{proof}

\section{Proof of the theorem 1.2}
The proof is by induction on number of boundaries. The induction begins at $b=4$.
\begin{lemma}
\label{begin}Let $A$ be a transcendental complex number. $\widetilde{PB_4} $ is dense in the algebraic group $SL(V_{0, 4; a, b, c, d})$ if $\dim{V_{0, 4; a, b, c, d}}\ge 2$
\end{lemma}
\begin{proof}

\begin{itemize}
\item The representation is irreducible. 

We adapt the notations of section 4.1 and claim the graph $C(V_{0, 4; a, b, c, d},<D_{\gamma_1}>,<D_{\gamma_2}>)$ is complete.

We still have the change of basis formula (1) in section 4.1. To prove the $v_i$ have a non-zero component in $W_j$, need to show coefficient
\begin{equation}
\left < 
                \begin{array}{ccc}
                 a & b & i \\
                 c& d  & j
                 \end{array} \right > =\sum_{z=\max{m_s}}^{\min{n_t}}\frac{\prod_{s, t}[n_s-m_t]!}{[a]![b]![c]![d]![i]![j]!} \frac{(-1)^z[z+1]!}{\prod_s[n_s-z]!\prod_t[z-m_t]!}\ne 0.
\end{equation}
Notice that when consider it as a rational function of $A$, the degree is different among its summands. Thus, the leading term only appeared in one of the summands. Since $A$ transcendental, it is not a root of any rational function,
so the coefficient is not zero. We proved the graph is complete and the representation is irreducible.

\item The image is infinite.

Consider two consecutive colors $v_i$ and $v_{i+2}$. Let $\lambda_i$ and $\lambda_{i+1}$ be the eigenvalues of $D_{\gamma_1}$ respectively. $\lambda_{i+1}/\lambda_{i}=-A^{4i+8}$. $A$ is not a root of unity. Thus, the image is infinite.

\item The closure of the image is $SL(V_{0, 4; a, b, c, d})$.

Assume the Zariski closure of the image is a semisimple lie algebra $G$. the action of $D_{\gamma_1}$ have the following eigenvalue set under some central extension:
$$X=\{A^{i^2}| i\in [\max\{|a-b|,|c-d|\}, \min\{|a+b|,|c+d|\}], i\equiv a-b \mod 2\}$$
Let's assume $i$ start at $i_0$ and end at ${i_n}$, There exist an element $a\in \g=Lie(G)$, the eigenvalues for $a$ acting on $V_{0, 4; a, b, c, d}$ is $\lambda_0=C+i_{0}^2, \lambda_1=C+(i_{0}+2)^2,...,\lambda_n=C+(i_n)^2$ for some constant $C$. We claim $\g$ is $\Sl_{n+1}$.

Notice in the weight space $\h^*$, $\lambda_0$ is the highest weight under some choice of simple roots. $\lambda_0-\lambda_1$ satisfying the following properties:
\begin{enumerate}
\item $\lambda_0-\lambda_1$ is not repeative, that is, $\lambda_0-\lambda_1\ne\lambda_k-\lambda_j$ for any $(k, j)\ne (0, 1)$.
\item $\lambda_0-\lambda_1$ is indecomposable, that is $\lambda_0-\lambda_1$ can not be written as positive integral linear combination of $\lambda_0-\lambda_k$ for $k\ne 1$.
\end{enumerate}

By Speyer's lemma, $\lambda_0-\lambda_1$ is a root. Consider the action of $W_{\lambda_0-\lambda_1}$ on the weight diagram. It send $\lambda_k$ to $\lambda_k$ plus(or minus) some copy of $\lambda_0-\lambda_1$. Because $\lambda_0-\lambda_1$ is not repeative, $W_{\lambda_0-\lambda_1}$ interchanges $\lambda_0$ and $\lambda_1$ and fix all other $\lambda_k$.

Then we prove for all $k$, $\lambda_0-\lambda_k$ is indecomposable. Assume otherwise, say $\lambda_0-\lambda_k=\sum_t\alpha_t(\lambda_0-\lambda_t)$. the decomposition is perserved by the Weyl group action. We have $\lambda_1-\lambda_k=\sum_t\alpha_t(\lambda_1-\lambda_t)$. These two equation together implies $\sum_t\alpha_t =1$. Thus $\lambda_0-\lambda_k$ is indecomposable for all $k$.

Use Speyer's lemma again, we know all $\lambda_0-\lambda_k$ are roots. Consider the action of $W_{\lambda_0-\lambda_k}$. It interchanges $\lambda_0$ and $\lambda_k$. Otherwise, $\lambda_0-W_{\lambda_0-\lambda_k}(\lambda_0)$ will be multiple of $\lambda_0-\lambda_k$. The action also have to fix all other $\lambda_i$. Assume otherwise, it send $\lambda_i$ to $\lambda_j$. then $\lambda_0-\lambda_j=\lambda_0-\lambda_i+(\lambda_i-\lambda_j)$ which is equal to $\lambda_0-\lambda_i$ plus multiple of $\lambda_0-\lambda_k$, so the indecomposable condition is contradicted.

${W_{\lambda_0-\lambda_k}}$ generate $S_{n+1}$. The lie algebra $\g$ have an $(n+1)$-dimensional irreducible representation and has $S_{n+1}$ as a quotient of a subgroup of its Weyl group. $\g$ have to be $\Sl_{n+1}$.
\end{itemize}
\end{proof}

To apply theorem \ref{surj}, we need the following proposition about irreducibility.

\begin{proposition}For $A$ any transcendental number, $V_{0, b;\vec{c}}$ is irreducible as a representation of $PB_{b}$.
\end{proposition}
\begin{proof}The proof is identical to lemma \ref{3}. Notice that we find a element in the intersection of $V_i$ and $W_j$. In transcendental case, this means we have directed edges of both direction.
\end{proof}

The following lemma gives a criterion that for some specific decomposition, the components is not isomorphic nor dual to each other.

\begin{lemma}
\label{noniso}Let $A$ be a transcendental number, and $c_1, c_2...c_n$ be a sequence of fixed non-decreasing natural numbers and $n\ge 3$. Then there is no pair of elements in $\{V_{0, n+1; a, c_1,..c_n}| a\le 2c_1\}$ are isomorphic or dual to each other.
\end{lemma}

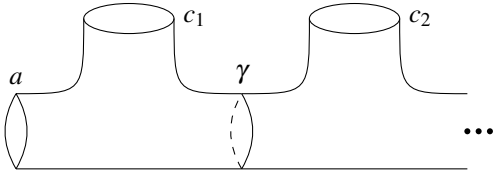
\begin{tpic}
\draw (-1.5, 0) to (4.5, 0);
\draw (-1.5, 0) to[bend right] (-1.5, 1)node[above]{$a$};
\draw (-1.5, 0) to[bend left] (-1.5, 1);
\draw (1.5, 0) to[bend right]  (1.5, 1)node[above]{$\gamma$};
\draw[dashed] (1.5, 0) to[bend left] (1.5, 1);
\draw (0, 2) ellipse (0.6 and 0.2);
\draw (1.5, 1) ..controls (0.6, 1) and (0.6, 1).. (0.6, 2) node[right]{$c_1$};
\draw (-1.5, 1) ..controls (-0.6, 1) and (-0.6, 1).. (-0.6, 2);
\draw (-1.5, 1) to (-1.5, 1);

\draw (4.5, 1) ..controls (3.6, 1) and (3.6, 1).. (3.6, 2) node[right]{$c_2$};
\draw (1.5, 1) ..controls (2.4, 1) and (2.4, 1).. (2.4, 2);
\draw (3, 2) ellipse (0.6 and 0.2);

\fill (4.5, 0.5)circle (0.04);
\fill (4.65, 0.5)circle (0.04);
\fill (4.8, 0.5)circle (0.04);

\end{tpic}

\begin{proof}
Consider the Dehn twist around the curve $\gamma$ that bounds the first two boundaries. The eigenvalue set will be 
$E_a=\{A^{i^2}|\max\{c_n-\sum_{k=2}^{n-1}c_k,|a-c_1|\}\le i\le\min\{a+c_1, \sum_{k=2}^{n}c_k\}\}$ up to some central extension. For different $a$, we claim either the cardinality of the eigenvalue set will be different or the set of ratios of eigenvalues are different. Set of ratios of eigenvalues is an invariant of central extensions, so the claim implies the lemma.

Suppose $a, a'\le c_1$ such that $E_a$ and $E_a'$ have the same cardinality. Then the interval for $i$ will be of the same length. To make the set of ratios of eigenvalues the same, the interval should start and end at the same place. Since we assumed $a, a'\le 2c_1\le 2c_2$, so at least one of the boundary of the interval is determined by $a$ or $a'$. Thus we proved the claim.
\end{proof}

Now we can give the proof of theorem \ref{dense} by induction:
\begin{proof}
The case number of boundary components $b=4$ is proved in lemma \ref{begin}.

Suppose the Zariski closure of the image of $PB_i$ in $PSL(V_{0, i; \vec{c}})$ is surjective for all $i\le n$ and coloring $\vec{c}$. we prove $PB_{n+1}$ has a dense image in $PSL(V_{0, n+1; \vec{c'}})$ for any $\vec{c'}$.

Let $G$ be the Zariski closure of $PB_{n+1}$. $G$ contains elements of infinite order, so $G$ must be of 
positive dimension. $G$ is generated by 1-dimensional subgroups that densely generated by Dehn twists, so $G$ is connected.

Without loss of generality, assume $\vec{c'}=(c_0, c_1, c_2...c_n)$ such that $c_i\le c_j$ if $i\le j$. Then we can have a decomposition of $V_{0, n+1; \vec{c'}}$ by restricting to a subgroup $H=\M(S')\cong PB_n$.
$$V_{0, n+1; \vec{c'}}=\bigoplus_{i=c_1-c_0}^{c_0+c_1}V_{0, n; i, c_2, c_3,...c_n}$$
According to lemma \ref{noniso}, all summands are not isomorphic nor dual to each other. 

To apply theorem \ref{surj}, we prove the dimension of the summands satisfying (2) of theorem \ref{surj}.

If the number of boundary component $b\ge 4$, the dimension of the vector spaces $V_{0, b; \vec{c}}$ have dimension $1$ if and only if one of the colors is equal to the summation of colors on all other boundaries. In our case, $b\ge 5$, $n\ge 4$, $i\le c_0+c_1$. $\dim V_{0, n; i, c_2, c_3,...c_n}=1$ if and only if $i=c_n-c_{n-1}-...-c_2$, so it happens at most once.

Next, we consider the 2-dimensional summands. We fix a uni-trivalent tree $\Gamma$ with $n$ boundary vertex colored by $\vec{c}$. If $\dim V_{0, n;\vec{c}}=2$, then we have 2 admissible coloring for $\Gamma$. Fix an edge $e$ that have different colors $a$ and $a+2$ in the two different admissible colorings. Cut the tree at $e$, the tree $\Gamma$ split to $\Gamma_1$ and $\Gamma_2$, and coloring $e$ by $a$ and $a+2$ both give unique admissible coloring for $\Gamma_1$ and $\Gamma_2$. This implies $\Gamma_1$ and $\Gamma_2$ can only have $3$ boundary vertices and $n=4$. When $n=4$, we can check by hand that at most for only one $i\le c_0+c_1$, $\dim V_{0, n; i, c_2, c_3,...c_n}=2$. 

By theorem \ref{surj}, $G$ surjects onto $PSL(V_{0, n+1; \vec{c'}})$.

\end{proof}

\bibliographystyle{plain}
\bibliography{irr,qa,gt}

\end{document}